\documentclass[english]{amsart}

\usepackage{babel}
\usepackage{amstext}
\usepackage{amsmath}
\usepackage{amsfonts}
\usepackage{latexsym}
\usepackage{amssymb}
\usepackage{ifthen}
\usepackage{xypic}
\xyoption{all}
\usepackage{hyperref}

\newcommand{\KH}{{\mathbb H}}

\theoremstyle{plain}

\newtheorem{thm}{Theorem}[section]

\newtheorem{pro}[thm]{Proposition}

\newtheorem{cor}[thm]{Corollary}
\newtheorem{con}[thm]{Conjecture}

\theoremstyle{definition}

\newtheorem{dfn}[thm]{Definition}

\newtheorem{rem}[thm]{Remark}

\newtheorem{exa}[thm]{Example}

\theoremstyle{remark}

\newtheorem{que}[thm]{Question}

\newcommand{\Z}{\mathbb{Z}}

\newcommand{\C}{\mathbb{C}}
\newcommand{\R}{\mathbb{R}}
\newcommand{\Q}{\mathbb{Q}}
\newcommand{\PS}{\mathbb{P}}

\newcommand{\OO}{\mathcal{O}}

\DeclareMathOperator{\inte}{int}

\DeclareMathOperator{\im}{Im}

\DeclareMathOperator{\Supp}{Supp}

\DeclareMathOperator{\GL}{GL}
\DeclareMathOperator{\Hom}{Hom}

\DeclareMathOperator{\Aut}{Aut}
\DeclareMathOperator{\Herm}{Herm}
\DeclareMathOperator{\PHerm}{PHerm}
\DeclareMathOperator{\End}{End}
\DeclareMathOperator{\Bir}{Bir}

\DeclareMathOperator{\Pic}{Pic}

\DeclareMathOperator{\NEb}{\overline{\mathrm{NE}}}

\DeclareMathOperator{\Nef}{Nef}
\DeclareMathOperator{\Amp}{Amp}

\DeclareMathOperator{\Eff}{\mathrm{Eff}}
\DeclareMathOperator{\Mov}{Mov}
\DeclareMathOperator{\B}{Big}

\newcommand{\id}{{\rm id}}

\newcommand\sE{{\mathcal E}}

\newcommand\sF{{\mathcal F}}

\newcommand\sI{{\mathcal I}}

\newcommand\sO{{\mathcal O}}

\newcommand\bQ{{\mathbb Q}}
\newcommand\bN{{\mathbb N}}

\newcommand\bP{{\mathbb P}}

\title {The Morrison-Kawamata Cone Conjecture and Abundance on Ricci flat manifolds} 

\author{Vladimir Lazi\'c}
\address{Mathematisches Institut, Universit\"at Bonn, Endenicher Allee 60, 53115 Bonn, Germany}
\email{lazic@math.uni-bonn.de}

\author{Keiji Oguiso}
\address{Graduate School of Mathematical Sciences, University of Tokyo, Komaba, Meguro, Tokyo, 153-8914, Japan and Korea Institute for Advanced Study, Hoegiro 87, Seoul, 130-722, Korea}
\email{oguiso@ms.u-tokyo.ac.jp}

\author{Thomas Peternell}
\address{Mathematisches Institut, Universit\"at Bayreuth, 95440 Bayreuth, Germany}
\email{thomas.peternell@uni-bayreuth.de}

\thanks{All authors were partially supported by the DFG-Forschergruppe 790 ``Classification of Algebraic Surfaces and Compact Complex Manifolds". Lazi\'c was supported by the DFG-Emmy-Noether-Nachwuchsgruppe ``Gute Strukturen in der h\"oherdimensionalen birationalen Geometrie". Oguiso was supported by JSPS Grant-in-Aid (S) No 25220701, JSPS Grant-in-Aid (S) No 22224001, JSPS Grant-in-Aid (B) No 22340009, and by KIAS Scholar Program.}

\begin{document}

\begin{abstract}
The aim of this survey paper is threefold: (a) to discuss the status of the Morrison-Kawamata cone conjecture, (b) to report on recent developments towards the Abundance Conjecture, and (c) to discuss the nef line bundle version of the Abundance Conjecture on $K$-trivial varieties.
\end{abstract}

\maketitle
\setcounter{tocdepth}{2}

\tableofcontents 

\section{Introduction} 

Since the fundamental work of S.\ Mori starting in the late 1970's, the study of the closed cone of ample divisor -- often called the nef cone --  and its dual, the closed cone of curves of a projective manifold $X$, is one of the cornerstones of higher-dimensional complex algebraic geometry, in particular in connection with the canonical bundle $K_X$. By now traditionally, the cone of curves is denoted by $\NEb(X)$. The basic {\it Cone Theorem} says that the $K_X$-negative part of $\NEb(X)$ is locally rational polyhedral. Furthermore, ``extremal'' rational points on the boundary of the negative carry significant geometric information: 
they lead to morphisms $\varphi\colon X \to Y$ of the variety.  The Minimal Model Program then continues to study $Y$ instead of $X$, and eventually produces a fibre space with $K$-negative (Fano) fibres or a variety with nef canonical bundle. In that case the {\it Abundance Conjecture} predicts that some multiple of the canonical bundle is spanned by global sections. 
 
The $K_X$-trivial part of $\NEb(X)$ gets more complicated; we restrict our attention here to the case when the whole bundle $K_X$ is trivial, i.e.\ we consider Ricci-flat projective manifolds. In this case the Cone Theorem has a hypothetical counterpart, the {\it Morrison-Kawamata cone conjecture}. This conjecture basically says that there is a rational polyhedral cone which is a fundamental domain for the action of the automorphism group on the (not necessarily closed) cone of nef {\it effective} divisors. This conjecture is -- contrary to the $K_X$ negative case -- wide open, even in dimension three. Again, it is interesting to look at the boundary of the nef cone; one might even conjecture that the nef cone itself is locally rational polyhedral, so that the boundary contains many rational points. These rational points correspond to nef line bundles $L$ and a line bundle version of the Abundance Conjecture would say that a multiple of $L = K_X + L$ is spanned by global sections. 

The aim of this paper is therefore threefold:

\begin{itemize}
\item to discuss the status of the Morrison-Kawamata cone conjecture,
\item to report on recent developments towards the Abundance Conjecture,
\item to discuss the nef line bundle version of the Abundance Conjecture on $K$-trivial varieties.
\end{itemize} 

\section{Notation}\label{sec:notation}

Unless otherwise stated, we work in the category of projective varieties defined over $\C$. Let $X$ be a complex normal projective variety. As usual, let $N^1(X)$ be the N\'eron-Severi group generated by the classes of divisors, using numerical equivalence; let $N^1(X)_\Q = N^1(X) \otimes \Q$ and $N^1(X)_\R = N^1(X) \otimes \R$. Inside $N^1(X)_\R$ we have 
\begin{itemize} 
\item the  nef cone $\Nef(X)$, which is just the closure of the ample cone; 
\item the effective nef cone $\Nef^e(X) := \Nef(X) \cap \Eff(X)$  (not necessarily closed);
\item the movable cone $\overline{\Mov}(X)$; 
\item the effective cone $\Eff(X)$; and
\item the effective movable cone $\overline{\Mov}^e(X) := \overline{\Mov}(X) \cap \Eff(X)$.
\end{itemize}

Recall that a divisor $D$ is {\it movable} if for some positive number $m$ the linear system $\vert mD \vert$ has no components of codimension 1. A divisor $D$ whose class is in the closure of $\Eff(X)$ is called {\it pseudoeffective}.

\vskip .2cm 

The closed cone of curves will be denoted by $\NEb(X) \subseteq H_2(X,\R)$. For further notations in the context of the minimal model program we refer e.g.\ to \cite{KM98}. 

\vskip .2cm 

Suppose now that a finite-dimensional real vector space $V$ has a distinguished $\Q$-structure, i.e.\ there is a $\Q$-vector space $V_\Q$ such that $V = V_\Q \otimes_\Q \R$. Let $C \subseteq V$ be a convex cone. Then by definition, $C^+$ is the smallest convex cone in $V$ containing all the $\Q$-rational points of $\overline C$. In particular we can speak of $\Nef^+(X)$, $\Amp^+(X)$ and $\Mov^+(X)$. 

\vskip .2cm

The group of automorphisms of a normal projective variety will as usual denoted by $\Aut(X),$ whereas $\Bir(X)$ is the group of birational automorphisms. 

\vskip .2cm 

Given a nef divisor $L$, the {\it numerical dimension} $\nu(L) $ is given by 
$$ \nu(L) = {\rm max} \{k \ \vert  \ L^k \not\equiv  0 \}.$$

Furthermore, a locally free sheaf $\mathcal E$ is {\it nef}, if the line bundle $\sO_{\PS(\sE)}(1) $ is nef. 
 
\section{The Morrison-Kawamata Cone Conjecture} 
\subsection{Formulation and basics} 
 
First we fix some notation.
 
\begin{dfn} 
A projective or compact K\"ahler manifold is said to be Ricci-flat if $K_X \equiv 0$; equivalently, $c_1(X) = 0$ in $H^2(X,\R)$. 
\end{dfn} 
 
Due to Yau's solution of the Calabi conjecture when $c_1(X) = 0$, there exists indeed a K\"ahler metric on a compact K\"ahler manifold $X$ with $K_X \equiv 0$ whose Ricci curvature vanishes.  
 
We recall the following important Beauville-Bogomolov decomposition theorem: 
 
\begin{thm} \label{thm:Beauville} 
Any Ricci-flat compact K\"ahler (resp.\ projective) manifold $M$  has a finite \'etale Galois cover $\pi \colon \widetilde{M} \to M$ such that $\widetilde{M}$ is the product of Calabi-Yau manifolds, hyperk\"ahler manifolds (resp.\ projective hyperk\"ahler manifolds) and a complex torus (resp.\ a projective complex torus, i.e.\ an abelian variety).
\end{thm}

According to this theorem, the most important classes of projective Ricci flat manifolds are abelian varieties, projective hyperk\"ahler manifolds and Calabi-Yau manifolds. To be precise, a Calabi-Yau manifold $X$ of dimension $n$ in this context is a (simply connected) projective manifold with trivial canonical class such that $H^q(X,\OO_X) = 0$ for $1 \leq  q \leq n-1$.

\vskip .2cm 

The Morrison-Kawamata Cone Conjecture, shortly referered to as Cone Conjecture in the following, is stated as follows. 
 
\begin{con} \label{conj1}
Let $X$ be a projective Ricci flat manifold. 
\begin{enumerate}
\item There exists a rational polyhedral cone $\Pi$ which is a fundamental domain for the action of $\Aut(X)$ on $\Nef^e(X) = \Nef(X)\cap\Eff(X)$, in the sense that
$$\Nef^e(X)=\bigcup_{g\in\Aut(X)}g^*\Pi,$$ 
and $\inte\Pi\cap \inte g^*\Pi=\emptyset$ unless $g^*=\id$.
\item There exists a rational polyhedral cone $\Pi'$ which is a fundamental domain for the action of $\Bir(X)$ on $\overline{\Mov}^e(X) = \overline{\Mov}(X)\cap\Eff(X)$.
\end{enumerate}
\end{con}

There is also the following weaker form.
 
\begin{con}
Let $X$ be a projective Ricci-flat manifold. \label{conj2} 
\begin{enumerate}
\item There exists a (not necessarily closed) cone $\Pi$ which is a weak fundamental domain for the action of $\Aut(X)$ on $\Nef^e(X)$, in the sense that 
$$\Nef^e(X) =\bigcup_{g\in\Aut(X)}g^*\Pi,$$
$\inte\Pi\cap \inte g^*\Pi=\emptyset$ unless $g^*=\id$, and for every $g\in\Aut(X)$, the intersection $\Pi\cap g^*\Pi$ is contained in a rational hyperplane.
\item There exists a polyhedral cone $\Pi'$ which is a weak fundamental domain for the action of $\Bir(X)$ on $\overline{\Mov}^e(X)$.
\end{enumerate}
\end{con}

Sometimes versions of the Morrison-Kawamata Cone Conjecture with slightly different cones will be considered, such as $\Nef^+(X)$ instead of $\Nef^{e}(X)$. 

\subsection{Abelian varieties} 

Let $A$ be an abelian variety. As an abelian variety has no rational curves, we have $\Bir(A) = \Aut(A)$ and $\overline{\Mov}^e(A) = \Nef^e(A)$. Therefore the Cone Conjectures 3.2 (1) and (2) coincide for an abelian variety. The Cone Conjecture for an abelian variety is actually highly non-trivial, as $\Nef (A)$ looks quite circular. 

We have the following very satisfactory answer for the cone conjecture for abelian varieties:

\begin{thm} (Prendergast-Smith \cite{PS12}) \label{thm:Smith} 
The Cone Conjecture is true for any abelian variety.  
\end{thm} 

The most crucial part of the proof is the following structure Theorem \ref{thm:Ash} of the homogeneous self-dual cone, which we now explain. 

Let $C$ be a strict open convex cone in a finite dimensional real vector space $V$. $C$ is called {\it homogeneous} if the linear subgroup 
$$\Aut(C) := \{g \in \GL(V, \R)\mid g(C) = C \}$$
acts transitively on $C$. 

Let $V^*$ be the dual vector space of $V$. The dual cone $C^* \subseteq V^*$ is the interior of 
$$\{\ell \in V^*\mid \ell(v) \ge 0 \text{ for all } v \in V\}.$$
The cone $C$ is called  {\it self-dual} if $C^* = C$ under the identification $V = V^*$ with respect to {\it some} non-degenerate linear form $b(\cdot\,,\cdot)$ on $V$. 

\begin{exa} \label{exa:selfdual1} 
(1)  Let $K = \R$, $\C$ or $\KH$, where $\KH$ is the division ring of quaternions. Let $V:= V_g$ be the real vector space consisting of $g \times g$ hermitian matrices with entries in $K$ and let $C_g$ be the cone consisting of positive definite hermitian matrices in $V_g$. 
Then $C:= C_g$ is a self-dual homogeneous cone with respect to the bilinear form $b(A, B) := {\rm tr}(AB^*)$. Here $B^* = (\overline{b_{ji}})$ when $B = (b_{ij})$, i.e.\ the conjugate in $K$ of the transpose of $B$. 

\vskip .2cm \noindent 
(2) Let $V := \R^{1, n}$ be the $(n+1)$-dimensional real vector space equipped with the bilinear form defined by 
$$b((x_i)_{i=0}^{n}, (y_i)_{i=0}^{n}) = x_0y_0 -(x_1y_1 + \dots +x_ny_n).$$ 
Then the cone  
$$C := C_n := \Big\{(x_i)_{i=0}^{n}\mid x_0 > \sqrt{x_1^2 + \dots +x_n^2}\Big\}$$
is a self-dual homogeneous cone with respect to this bilinear form.
\end{exa}

Note that the vector spaces $V$ in these examples has a $\Q$-structure. Thus we may consider the cone $C^+$ introduced in Section \ref{sec:notation}. 

The following group-theoretic theorem due to Vinberg, Ash and Looijenga (see references in \cite{PS12}) is the core in the proof of Theorem \ref{thm:Smith}.

\begin{thm} \label{thm:Ash} 
Let $C$ be a self-dual homogeneous cone. Then the automorphism group $\Aut(C)$ is the group $G(\R)$ of real points of a reductive algebraic group $G$. If the identity component of $G$ is defined over $\Q$ (which holds in the examples above), then for any arithmetic subgroup $\Gamma$ of $G$ there is a rational polyhedral fundamental domain $\Delta \subseteq C^+$ for the action of $\Gamma$ on $C^+$. 
\end{thm}

By the Poincar\'e complete reducibility theorem, the abelian variety $A$ under discussion is isogenous to the product of  {\it simple} abelian varieties $A_i$, say $A$ is isogenous to 
$$A_1^{n_1} \times \ldots \times A_k^{n_k},$$ 
where $A_i$ are simple abelian varieties such that $A_i$ and $A_j$ are not isogenous. Then 
$$\End_\Q(A) := \End(A) \otimes_\Z\Q \simeq M_{n_1}\big(\End_\Q(A_1)\big) \times \dots \times M_{n_k}\big(\End_\Q(A_k)\big).$$
Here $\End_\Q(A_i)$ are finite dimensional division algebras over $\Q$ (as $A_i$ are simple), and $M_{n_i}\big(\End_\Q(A_i)\big)$ is the ring of $r \times r$ matrices with entries in $\End_\Q(A_i)$. Hence
$$\End_\R(A) := \End_\Q(A) \otimes_\Q\R \simeq \prod_i M_{r_i}(\R) \times \prod_j M_{s_j}(\C) \times \prod_k M_{t_k}(\KH)$$
by Albert's classification of finite dimensional $\Q$-division algebras with positive involution. Here the positive involution on the left hand side, the Rosati involution, corresponds to the conjugate transpose on the right hand side under the above isomorphism. 

The essential idea of the proof of Theorem \ref{thm:Smith} is to describe $\Aut(A)$ (resp.\ $N^1(A)$ and $\Amp(A)$) using the action on (resp.\ the subspaces of) the space $\prod_i M_{r_i}(\R) \times \prod_j M_{s_j}(\C) \times \prod_k M_{t_k}(\KH)$ to reduce the proof to Theorem \ref{thm:Ash}. This can be done as follows.  

Note that $\Aut(A) = \End(A)^\times$. Then, by definition, $\Aut(A)$ is an arithmetic subgroup of the algebraic group $\End_\R(A)^\times$ which is defined over $\Q$. 

Recall the following natural group homomorphism:
$$\Phi \colon \Pic(A) \to \Hom(A, \Pic^0(A)),\quad D \mapsto (x \mapsto t_x^*(D) \otimes D^{-1}),$$
where $t_x$ is the translation by $x \in A$. Choose any ample line bundle $L$ on $A$. It is a classical fact that $\Phi(L)$ is an isogeny from $A$ to ${\rm Pic}^0(A)$ and $\Phi(D) = 0$ in 
$$\Hom(A, \Pic^0(A)) \otimes_\Z \Q = \End_\Q(A)$$ 
exactly when $D \in \Pic^0(A)$. As $\Phi(L)$ is an isogeny, the inverse map $\Phi(L)^{-1}$ is also well-defined in $\End_\Q(A)$. Then we have a natural group homomorphism
$$\Pic(A) \to \End_\Q(A) \subseteq \End_\R(A),\quad D \mapsto \Phi(L)^{-1}\Phi(D),$$
which descends to the {\it injective} group homomorphism
$$\rho \colon N^1(A)  \to \End_\R(A) \simeq \prod_i M_{r_i}(\R) \times \prod_j M_{s_j}(\C) \times \prod_k M_{t_k}(\KH).$$
As one can easily guess (and prove by looking at the Rosati involution), we have the following expected descriptions of $N^1(A)$ and $\Amp(A)$ in $\prod_i M_{r_i}(\R) \times \prod_j M_{s_j}(\C) \times \prod_k M_{t_k}(\KH)$:
\begin{align*}
\rho\big(N^1(A)\big) &= \prod_i \Herm_{r_i}(\R) \times \prod_j \Herm_{s_j}(\C) \times \prod_k \Herm_{t_k}(\KH),\\
\rho\big(\Amp(A)\big) &= \prod_i \PHerm_{r_i}(\R) \times \prod_j \PHerm_{s_j}(\C) \times \prod_k \PHerm_{t_k}(\KH).
\end{align*}
Here ${\rm Herm}_{r}(K)$ is the space of $r \times r$ hermitian matrices and ${\rm PHerm}_{r_i}(K)$ is the set of positive $r \times r$ hermitian matrices. 

Now we can apply Theorem \ref{thm:Ash} to conclude. 

\subsection{Hyperk\"ahler manifolds}

Let $S$ be a projective K3 surface. Then $S$ is nothing but a $2$-dimensional projective hyperk\"ahler manifold. Even though a K3 surface $S$ often admits (infinitely many) smooth rational curves, one has 
$$\Aut(S) = \Bir(S),\quad \overline{\Mov}^e(S) = \overline{\Mov}^+(S) = \Nef^e(S) = \Nef^+(S),$$
as $S$ is a minimal surface such that any rational nef divisor is semiample. So, for projective K3 surfaces the two cone conjectures again coincide and can be formulated in terms of $\Nef^+(S)$. 

We have also the following fairly satisfactory answer due to Sterk for projective K3 surfaces and to Markman for projective hyperk\"ahler manifolds:

\begin{thm} \cite{St85,Mk11} \label{thm:SterkMarkman} 
\begin{enumerate}
\item The Cone Conjecture holds for any projective K3 surface $S$.
\item The movable Cone Conjecture for $\Mov^+(X)$ holds for any projective hyperk\"ahler manifold $X$. 
\end{enumerate}
\end{thm}

The ideas of the proofs of both statements are essentially the same and are somewhat similar to the abelian case. 

In the following,  the $\Z$-structure of the N\'eron-Severi group $N^1(X)$ of the hyperk\"ahler manifold $X$ is important. We denote 
$$\Bir(X)^* = \im\big(\Bir(X) \to O^+(N^1(X))\big)$$
and 
$$\Aut(X)^* = \im\big(\Aut(X) \to O^+(N^1(X))\big).$$
Here $O^+(N^1(X))$ is the orthogonal group of $N^1(X)$ preserving the positive cone, with respect to the Beauville-Bogomolov-Fujiki form. As $X$ is a hyperk\"ahler manifold, $f^* \in O^+(N^1(X))$ for $f \in \Bir(X)$. In the case of a K3 surface,  the Beauville-Bogomolov-Fujiki form is nothing but the intersection form.

To make the essential part clear, let us explain first how to proceed in the case of projective K3 surfaces, using the following version of the global Torelli theorem.

\begin{thm} \label{thm:TorelliK3} 
Let $S$ be a projective K3 surface. Then there is a finite index normal subgroup $\Gamma^0$ of $\Aut(S)^*$ and a finite index subgroup $\Gamma$ of $O(N^1(S))$ with a semi-direct product decomposition
$$\Gamma = \Gamma^0 \ltimes W(N^1(S)).$$
Here $W(N^1(S))$ is the reflection group generated by the effective $(-2)$-curves on $S$. Moreover, $\Amp^+(S)$ is a fundamental domain of the action of $W(N^1(S))$ on the $\Q$-rational hull $P(S)^+$ of the positive cone $P(S) \subseteq N^1(X)_\R$. 
\end{thm}

Let $O^+(N^1(S))$ be the orthogonal group of $N^1(S)$ preserving the positive cone. As $\Gamma$ is an arithmetic subgroup of $O(N^1(S)_\R)$ by definition, we can find a rational polyhedral fundamental domain $\Delta$ of the action of $\Gamma$ on $P^+(S)$ by Theorem \ref{thm:Ash}. As $W(N^1(S))$ is generated by the reflections with respect to hyperplanes, we can choose $\Delta$ so that $\Delta \subseteq \Amp^+(S)$, using the semi-direct product structure of Theorem \ref{thm:TorelliK3} and the definition of the fundamental domain. Then, again by the semi-direct product structure, $\Gamma$ is nothing but the fundamental domain of the action $\Gamma^0$ on $\Amp^+(S)$. As $\Gamma^0$ is a normal finite index subgroup of $\Aut(S)^*$, Theorem \ref{thm:SterkMarkman}(1) follows. 

\begin{rem} 
If $S$ is a very general K3 surface, then $S$ is non-projective. So the ample cone is empty but the K\"ahler cone is still rich. For this reason, one might expect a version of the Cone Conjecture for the K\"ahler cone in $H^{1,1}(S)$ for a non-projective K3 surface. However, this expectation is not met. In fact, if $S$ is very general, then $\rho(S) = 0$ and $\Aut(S) = \{\id_S\}$. Moreover, $S$ has no smooth rational curves and the K\"ahler cone of $S$ coincides with the positive cone in $H^{1,1}(S)$, which is completely circular. This is also the unique fundamental domain as $\Aut(S) = \{\id_S\}$. So the version of cone conjecture for the K\"ahler cone does not hold for a very general K3 surface. 
\end{rem}

For a projective hyperk\"ahler manifold $X$, we have a similar Torelli type theorem, not for $\Aut(X)$ but for $\Bir(X)$. For the statement, we say that an effective divisor $E$ on $X$ is exceptional if the matrix $\|q(E_i, E_j)\|_{i,j}$ is negative definite, where $\Supp E = \sum_i E_i$ is the decomposition into  irreducible components of $E$ and $q$ is the Beauville-Bogomolov-Fujiki form. Note that an exceptional divisor on a K3 surface (in this sense) is nothing but an effective sum of $(-2)$-curves. 

Recall that the positive cone on $X$ is defined to be the connected component of the cone of classes of divisors $D$  with $q(D,D) > 0$ containing the ample cone. 

The following formulation is due to Markman after an important work of Huybrechts and Verbitsky (see \cite{Mk11} for details):

\begin{thm} \cite{Mk11} \label{thm:TorelliHK} 
There is a finite index normal subgroup $\Gamma^0$ of $\Bir(X)^*$ and a finite index subgroup $\Gamma$ of $O(N^1(X))$ such that 
$$\Gamma = \Gamma^0 \ltimes W(N^1(X)).$$
Here $W(N^1(X))$ is the group generated by the reflections with respect to all the exceptional effective divisors $E$ on $X$. Moreover, $\overline{\Mov}^+(X)$ is a fundamental domain of the action of $W(N^1(X))$ on the $\Q$-rational hull $P(X)^+$ of the positive cone $P(X) \subseteq N^1(X)_\R$.  
\end{thm}

The miracle here is the inclusion $W(N^1(X)) \subseteq O(N^1(X))$ even though $q(E,E) \neq -2$ in general. By using Theorem \ref{thm:TorelliHK}, one proves Theorem \ref{thm:SterkMarkman}(2) in the exactly same manner as Theorem \ref{thm:SterkMarkman}(1). 

\begin{rem} \label{rem:TorelliHK} 
To reduce the second assertion in the last theorem for $\overline{\Mov}^e(M)$ from the corresponding assertion for $\overline{\Mov}^+(M)$, we would need a strong abundance type result for any nef rational divisor $D$ with $q(D, D) = 0$.    
\end{rem}

There seems no such group theoretic interpretation of the action of $\Aut(M)$. However, recently Markman-Yoshioka \cite{MY15} and Amerik-Verbitsky \cite{AV15} proved the following very interesting result on the Cone Conjecture for the nef cone as a special case of a more general result which reduces the problem essentially to a degree bound for special ``exceptional'' divisors (see \cite{MY15} for this general statement):

\begin{thm} \cite{MY15,AV15} \label{thm:deformHilb} 
Let $X$ be a projective hyperk\"ahler manifold. If $X$ is deformation equivalent to the Hilbert scheme of points on a K3 surface or a generalised Kummer variety. Then the Cone Conjecture \ref{conj1}(1) for $\Nef^+(X)$ holds. 
\end{thm}

Very recently Amerik-Verbitsky \cite{AV14} generalised this theorem to any projective hyperk\"ahler manifold $X$ with $b_2(X) \neq 5$. 

\subsection{Calabi-Yau Manifolds} 

In this subsection we discuss the Cone Conjectures \ref{conj1} and \ref{conj2} for Calabi-Yau manifolds. As mentioned before, a \emph{Calabi-Yau manifold} is a simply connected projective manifold $X$ with trivial canonical bundle such that $H^q(X,\OO_X) = 0$ for $1 \leq q \leq \dim X -1$. A \emph{weak Calabi-Yau manifold} is a projective manifold $X$ with trivial canonical bundle such that $H^1(X,\OO_X) = 0$. 

A main source of examples is the following theorem due to Koll\'ar, see \cite{Bor91}.

\begin{thm}  \label{thm:hypersurface} 
Let $Z$ be a Fano manifold of dimension at least $4$, and let $X \subseteq |{-}K_Z |$ be a smooth divisor. Then the inclusion $j\colon X \to Z $ induces the bijection
$$ j_*\colon \NEb(X) \to \NEb(Z). $$
In particular, $\rho(X) = \rho(Z)$. 
\end{thm} 

When $\Aut (X)$ is finite, Conjecture \ref{conj1}(1) says that $\Nef (X) \cap \Eff (X)$ is a rational polyhedral cone. In particular,  $\Nef (X) \cap \Eff (X)$ is closed and therefore  $\Nef (X) \cap \Eff (X) = \Nef(X).$ Thus, if $L$ is a nef divisor, then $L$ is effective (and potentially $L$ is semiample, see Section \ref{sec:abundance}). 

If $X$ is a smooth anticanonical section of dimension at least three in a Fano manifold $Z$, then by Theorem \ref{thm:hypersurface} the cone  $\NEb (X) $ is indeed a rational  polyhedral cone, since $\NEb(Z) $ is rational polyhedral.

Related to this, the following theorem of Kawamata \cite{Kaw97,KKL12} is highly relevant: 

\begin{thm} 
Let $X$ be a normal projective $\Q$-factorial variety with terminal singularities and with trivial canonical class. Then the cones $\Nef (X)$ and $\overline{\Mov} (X)$ are locally rationally polyhedral in the big cone $\B (X)$.
\end{thm} 

In view of the Cone conjecture, it is important to study $\Aut (X)$ and $\Bir (X)$ for weak Calabi-Yau manifolds. Clearly, both of these groups are discrete. The following fact is very useful:

\begin{pro} \label {pro:basic} 
Let $X$ be a Calabi-Yau manifold and let $G$ be a subgroup of $\Bir (X)$. Assume that there is an ample line bundle $\mathcal L$ on $X$ such that $f^*\mathcal L \simeq \mathcal L$ for all $f \in G$. Then $G$ is finite and it is a subgroup of $\Aut (X)$. 
\end{pro} 

If $\rho(X) = 1$ (this case is of course not interesting from the point of view of the Cone conjecture), then it is classically known that $\Aut (X) = \Bir(X)$, $\Aut (X)$ is finite by Proposition \ref{pro:basic}, and non-trivial finite groups really occur, see \cite{Og14}. It is of its own interest to study what kind of finite groups appear as the automorphism groups of Calabi-Yau manifolds of $\rho(X) = 1$. In this direction, the complete classification of the automorphism groups of smooth quintic Calabi-Yau threefolds -- the most basic Calabi-Yau manifolds of $\rho(X) = 1$ -- was carried out by \cite{OY15}. However, from a cone-theoretic point of view, the first interesting case is when $\rho(X) = 2$, which was analysed in \cite{Og14}:

\begin{thm} \label{thm:Og-aut}
Let $X$ be a weak Calabi-Yau manifold of dimension $n$ with $\rho(X) = 2$. Then $\Aut (X)$ is finite provided either $n$ is odd, or $n$ is even and there is no real number $c\in\R$ and no real quadratic form $q_X$ on $N^1(X)_\R$ such that $x^n = c q_X(x)^{n/2}$ for all $x \in N^1(X)_\R$. 
\end{thm} 

There are indeed plenty of interesting examples of Calabi-Yau threefolds with $\rho(X) = 2$. One might speculate that such a quadratic form $q_X$ can never exist on a Calabi-Yau manifold: otherwise, $q_X$ would resemble the Beauville-Bogomolov form on a hyperk\"ahler manifold. Another hint  that $\Aut (X) $ must be finite if $\rho(X) = 2$ is provided by the following observation. If $\Aut (X)$ is infinite, then 
$$ c_{i_1}(X)  \cdot \ldots \cdot c_{i_r}(X) = 0$$
for all positive integers $i_j$ such that $i_1 + \dots + i_r = n - 1$. In particular, $c_{n-1}(X) = 0$. One might wonder whether a simply connected projective manifold $X$ of dimension $n$ with $c_{n-1}(X) = 0$ must be symplectic. 

In order to discuss Theorem \ref{thm:Og-aut}, we consider the natural action 
$$ r\colon \Bir (X) \to \GL(N^1(X))$$
on $N^1(X)$. Let $\mathcal B(X)$, respectively $\mathcal A(X)$ be the image of $\Bir (X)$, respectively $\Aut (X)$, via $r$. The main part of the proof of Theorem \ref{thm:Og-aut} is to show that every element in $\mathcal A(X)$ has order $2$. Since $r$ has finite kernel, Theorem \ref{thm:Og-aut} then follows from Burnside's theorem. 

By \cite{LP13}, we have:

\begin{thm}  \label{thm:LP1}
Let $X$ be a Calabi-Yau manifold with $\rho(X) = 2$. Then either $|\mathcal A(X)|\leq2$, or $|\mathcal A(X)|$ is infinite; and either $|\mathcal B(X)|\leq2$, or $|\mathcal B(X)|$ is infinite.
\end{thm}

The consequences for the Cone Conjectures can be summarized as follows. 

\begin{thm} \label{thm:LP2}
Let $X$ be a Calabi-Yau manifold with $\rho(X) = 2$. Then
\begin{enumerate}
\item if the group $\Bir(X)$ is finite, then the weak Cone Conjecture holds on $X$;
\item if the group $\Bir(X)$ is infinite, then the Cone Conjecture holds on $X$.
\end{enumerate}
\end{thm}

The assumption that $\rho(X) = 2$ is heavily used: the cone $\NEb(X)$ has exactly two boundary rays, say $\ell_1$ and $\ell_2$. So if $f \in \Aut (X)$, then either the $\ell_j$ are invariant under $f^*$, or $f^*$ maps $\ell_1$ to $\ell_2$ and vice versa. When $f$ is a non-regular birational automorphism, it is necessary to consider the movable cone. 

\begin{exa}  
We describe an example from \cite[Proposition 1.4]{Og14}, exhibiting a Calabi-Yau threefold $X$ with $\rho(X) = 2$ and $\Bir (X)$ infinite. The manifold $X$ is a complete intersection of three general hypersurfaces in $\PS^3 \times \PS^3$ of types $(1,1), (1,1)$ and $(2,2)$. The two projections of $\PS^3 \times \PS^3$ yield morphisms
$$ p_i\colon X \to \PS^3, \quad i=1,2.$$
These morphisms have degree $2$, and explicit calculations show that their Stein factorisations are small contractions, i.e.\ they contract finitely many (smooth rational) curves. Then the line bundles 
$$L_i = p_i^*\OO_{\PS^3}(1) $$ 
are big and nef, but not ample, hence they define boundary rays of $\NEb(X)$. In order to describe $\Bir (X)$, let 
$$\tau_i\colon X \dasharrow  X$$ 
be the covering involutions induced by the maps $p_i$. Then $\Bir (X)$ is generated by $\Aut (X)$ and by the birational automorphisms $\tau_1, \tau_2$. Furthermore, $\tau_1\tau_2$ is of infinite order, hence $\Bir (X)$ is infinite. It can also be shown that both boundary rays of the movable cone $\overline{\Mov} (X)$ are irrational.
\end{exa} 

If $\rho(X) = 3$, the only known result, proved in \cite{LOP13}, concerns dimension three: 

\begin{thm} \label{thm:LOP}
Let $X$ be a weak Calabi-Yau threefold with $\rho(X) = 3$. Then either $\Aut (X)$ is finite, or $\Aut (X)$ is almost abelian of rank $1$, i.e.\ $\Aut (X) \simeq \Z$ up to finite kernel and cokernel. 
\end{thm} 

Currently, no example of $X$ is known with $\Aut (X)$ infinite in the context of Theorem \ref{thm:LOP}. However, Borcea \cite{Bor91a} exhibited an example of a Calabi-Yau threefold with $\rho(X) = 4$ with infinite automorphism group. For examples with large Picard number, see \cite{GM93,OT15}. We remark that if $\pi_1(X) $ is infinite, then $\Aut (X)$ is finite, see \cite{LOP13} for references. 

\vskip .2cm 

Wilson \cite{Wi94} related the Chern class $c_2(X)$ to $\Aut (X)$: 

\begin{thm} 
Let $X$ be a weak Calabi-Yau threefold such that $c_2(X) $ lies in the interior of $\NEb(X),$ i.e.\ $D \cdot c_2(X) > 0 $ for all nef $\R$-divisors $D \neq 0$. Then $\Aut (X)$ is finite. 
\end{thm} 

If $X$ is a Calabi-Yau threefold which is a smooth anticanonical section of a Fano $4$-fold, then it was checked \cite{OP98} that $c_2(X) > 0$, hence:

\begin{thm} 
Let $Z$ be a Fano $4$-fold and let $X \in | {-}K_Z| $ be a smooth divisor. Then $\Aut(X) $ is finite. 
\end{thm} 

We describe another interesting series of examples. Let $P_{n+1} := (\bP^1)^{n+1}$ for some integer $n \geq 3$ and let $X_n$ be a general element of $|{-}K_{P_{n+1}}|$, i.e.\ a generic hypersurface of multidegree $(2, 2, \ldots, 2)$ in the $(n+1)$-dimensional Fano manifold $P_{n+1}$. Then $X_n$ is a Calabi-Yau manifold of dimension $n$ with $\rho(X_n) = n+1$. As already observed, $\Nef(X_n) \simeq \Nef(P_{n+1})$ and Conjecture 3.2 (1) is trivially holds. Moreover:

\begin{thm} \cite{CO15} \label{thm:CantatOguiso}
\begin{enumerate} 
\item  $\Aut(X_n)$ is trivial.
\item $\Bir(X_n)$ is the free product of the $n+1$ involutions $\iota_i$, $1 \leq i \leq n+1$, of $X_n$.
\item The movable Cone Conjecture \ref{conj1}(2) holds for $X_n$.
More precisely, the nef cone $\Nef(X_n)$ is a fundamental domain for the action of $\Bir(X_n)$ on $\overline{\Mov}^e (X_n)$.  
\end{enumerate} 
\end{thm} 

This seems to be the only series of examples in any dimension $\ge 3$ for which the movable Cone Conjecture \ref{conj1}(2) has been checked. 

To explain the proof, consider the involution $\iota_i$, defined as the covering involution of the projection $X_n \to (\bP^1)^n_{i}$, where $(\bP^1)^n_{i}$ is obtained from $(\bP^1)^{n+1}$ by deleting the $i$-th factor. One of the essential parts of the proof is to observe that the $n+1$ birational involutions $\iota_i \in \Bir(X_n)$ are at the same time all possible flops of $X_n$. Then one can apply the following fundamental theorem due to Kawamata:

\begin{thm} \cite{Ka08} \label{thm:Kawamata} 
Let $Y$ be a terminal minimal model. Then any $f \in \Bir(Y)$ is decomposed as 
$$f = \varphi \circ \gamma_{m-1} \circ \cdots \circ \gamma_0,$$ 
where $Y_0 = Y = Y_m$, $\gamma_i \colon Y_i \dashrightarrow Y_{i+1}$ are flops between minimal models $Y_i$ and $Y_{i+1}$, and $\varphi \in \Aut(X)$. 
\end{thm}

It is natural to ask the following special but interesting case of the movable Cone Conjecture, as a possibly tractable generalisation of Theorem \ref{thm:CantatOguiso}:

\begin{que}  \label{que:CantatOguiso}
Let $Z$ be a smooth Fano manifold of dimension at least $4$. Assume that the linear system $|{-}K_Z|$ is free. Does the movable Cone Conjecture hold for a general $X \in |{-}K_Z|$?
\end{que} 

\section{Abundance Conjecture} \label{sec:abundance}

\subsection{Abundance for klt pairs}\label{subsection:4.1}

The Abundance Conjecture is one of the most important open problems in higher dimensional geometry of projective varieties in characteristic zero. Its importance stems from the fact that the full Minimal Model Program would imply that, birationally, all varieties are built of varieties whose curvature is either positive, flat or negative.

Recall that given a $\mathbb Q$-factorial projective klt pair $(X,\Delta)$, the Minimal Model Program (MMP) predicts that either $(X,\Delta)$ has a birational model which admits a Mori fibration, or $(X,\Delta)$ has a birational model $(X',\Delta')$ with klt singularities such that $K_{X'}+\Delta'$ is nef; the pair $(X',\Delta')$ is called a \emph{minimal model} of $(X,\Delta)$. The following \emph{abundance conjecture} then predicts that all minimal models are \emph{good model}. 

\begin{con} 
Let $(X,\Delta)$ be a klt pair. If $K_X + \Delta $ is nef, then $K_X + \Delta $ is semiample, i.e.\ some multiple $m(K_X+\Delta)$ is basepoint free. 
\end{con} 

In particular, if abundance holds on a minimal model $(X,\Delta)$, then there exists a morphism with connected fibres $f\colon X\to Z$ such that $K_X+\Delta\sim_\bQ f^*A$ for some ample $\bQ$-divisor $A$ on $Z$. Furthermore, by the main result of \cite{Amb05a}, there exists a $\bQ$-divisor $\Gamma\geq0$ on $Z$ such that the pair $(Z,\Gamma)$ is klt and $K_X+\Delta\sim_\bQ f^*(K_Z+\Gamma)$. Note that, in general, $\Gamma\neq0$ even when $\Delta=0$; this explains why it is very natural to consider the MMP for pairs and not only for varieties. In particular, the Abundance Conjecture allows an inductive approach to the birational geometry of algebraic varieties.

The ``classical'' case is when $\Delta = 0$ and $X$ has terminal singularities. In that case, the Abundance Conjecture was proved for threefolds by Miyaoka and Kawamata in \cite{Miy87,Miy88a,Miy88b,Kaw92}, and abundance for canonical fourfolds is known when $\kappa(X,K_X)>0$ by \cite{Kaw85}. The corresponding generalisation to threefold klt pairs was established in \cite{KMM94}. In arbitrary dimension, until very recently the only results were the basepoint free theorem, which proves abundance for klt pairs of log general type \cite{Sho85,Kaw85b}, and abundance for varieties with numerical dimension $0$, see \cite{Nak04}. 

\medskip

The problem often splits into two parts: 
\begin{itemize}
\item \emph{nonvanishing}: showing that $\kappa (X,K_X + \Delta) \geq 0$, and
\item \emph{semiampleness}: showing that if $\kappa (X,K_X +\Delta) \geq 0 $, then $K_X + \Delta$ is semiample. 
\end{itemize}
The proof of nonvanishing for threefolds by Miyaoka \cite{Miy87,Miy88a} is an ingenious application of his inequality of Chern classes to the Riemann-Rich formula, together with the use of semistability of vector bundles and the Yang-Mills theory developed by Donaldson; a clear overview of these ideas can be found in \cite{MP97}. On the other hand, the proof of semiampless for threefolds by Miyaoka and Kawamata \cite{Miy88b,Kaw92} is a careful analysis of a section $D\in H^0(X,mK_X)$ and uses deformation theory to improve how components of $D$ sit inside of $X$. Unfortunately, both of these proofs use surface and threefold geometry crucially, and cannot be generalised to higher dimensions in a straightforward manner.

\medskip

In \cite{LP16}, a new approach to abundance in higher dimensions is introduced. The main idea is that the growth of global sections of the sheaves $\Omega_X^{[q]}\otimes\OO_X(mK_X)$ should correspond to the growth of sections of $\OO_X(mK_X)$, where $\Omega_X^{[q]}=(\bigwedge^q \Omega^1_X)^{**}$ is the sheaf of reflexive $q$-differentials on a normal variety $X$. In the context of nonvanishing, the main technical result of \cite{LP16} is the following (a terminal variety being a shorthand for a normal variety with at most terminal singularities):

\begin{thm}\label{thm:nonvanishing}
Let $X$ be a terminal projective variety of dimension $n$ with $K_X$ pseudoeffective. Assume either 
\begin{enumerate}
\item[(i)] the existence of good models for klt pairs in dimensions at most $n-1$, or 
\item[(ii)] that $K_X$ is nef and $\nu(X,K_X)=1$. 
\end{enumerate}
Assume that there exists a positive integer $q$ such that
$$ h^0\big(X,\Omega^{[q]}_X \otimes \OO_X(mK_X)\big)>0$$
for infinitely many $m$ such that $mK_X$ is Cartier. Then $\kappa(X,K_X)\geq0$.
\end{thm}

There are two main new inputs in the proof of this result. The first is using a kind of stability of the cotangent bundle \cite{CP11,CP15}:

\begin{thm}\label{thm:CP11}
Let $(X,\Delta)$ be a log smooth projective pair, where $\Delta$ is a reduced divisor. Let $\Omega_X^1(\log\Delta)^{\otimes m}\to\mathcal Q$ be a torsion free coherent quotient for some $m\geq1$. If $K_X+\Delta$ is pseudoeffective, then $c_1(\mathcal Q)$ is pseudoeffective.
\end{thm}

This result generalises Miyaoka's generic semipositivity theorem \cite{Miy87a}. It is used to show that the assumptions of Theorem \ref{thm:nonvanishing} imply the existence of a pseudoeffective divisor $F$ and of Weil divisors $N_m\geq0$ for infinitely many $m$ with
$$N_m\sim mK_X-F.$$

The second input is running a Minimal Model Program with scaling for a carefully chosen pair $(X,\varepsilon N_k)$, to show the existence of a birational contraction $X\dashrightarrow X'$ and a fibration $f\colon X'\to Y$ to a lower-dimensional variety $Y$ and of a big $\Q$-divisor $D$ on $Y$ such that $K_{X'}\sim_\Q f^*D$. This produces the desired result. Similar techniques work in the context of nef line bundles on varieties of Calabi-Yau type, and we give more details in the following subsections.

\medskip

In order to apply Theorem \ref{thm:nonvanishing},  we need to consider singular metrics on $\sO_X(L)$. First we recall the definitions, see e.g.\ \cite{DPS01,Dem01}.

\begin{dfn} 
Let $X$ be a normal projective variety and let $D$ be a $\Q$-Cartier divisor on $X$. Then $D$, or $\OO_X(D)$, has a metric with \emph{analytic singularities} and semipositive curvature current, if there exists a positive integer $m$ such that $mD$ is Cartier and if there exists a resolution of singularities $\pi\colon Y \to X$ such that the line bundle $\pi^*\sO_X(mD)$ has a singular metric $h$ whose curvature current is semipositive and such that the local plurisubharmonic weights $\varphi$ of $h$ are of the form
$$ \varphi = \sum \lambda_j \log \vert g_j \vert + O(1),$$
where $\lambda_j$ are positive real numbers, $O(1)$ is a bounded term, and the divisors $D_j$ defined locally by $g_j$ form a simple normal crossing divisor on $Y$. We then have
$$\textstyle\mathcal I(h^{\otimes m})=\sO_Y\big(-\sum\lfloor m\lambda_j\rfloor D_j\big)\quad\text{for every positive integer }m,$$
where $\mathcal I(h^{\otimes m})$ is the multiplier ideal associated to $h^{\otimes m}$. If all $\lambda_j$ are rational, then $h$ has \emph{algebraic singularities}. Further, $\mathcal O_X(D)$  is \emph{hermitian semipositive} if $\pi^*\sO_X(mD)$ has a smooth hermitian metric $h$ whose curvature $\Theta_h(D)$ is semipositive. 
\end{dfn}

The following result is the Hard Lefschetz Theorem from \cite{DPS01}, and it is crucial for applications of Theorem \ref{thm:nonvanishing}.

\begin{thm}\label{thm:DPS}
Let $X$ be a compact K\"ahler manifold of dimension $n$ with a K\"ahler form $\omega$. Let $\mathcal L$ be a pseudoeffective line bundle on $X$ with a singular hermitian metric $h$ such that $\Theta_h(\mathcal L) \geq 0$. Then for every nonnegative integer $q$ the morphism
\[
\xymatrix{ 
H^0\big(X,\Omega^{n-q}_X\otimes\mathcal  L\otimes\mathcal I(h)\big) \ar[r]^{\omega^q\wedge\bullet} & H^q\big(X, \Omega^n_X\otimes \mathcal L\otimes\mathcal I(h)\big)
}
\]
is surjective.
\end{thm}

As an immediate application of Theorem \ref{thm:nonvanishing}, following \cite{DPS01} we show the following:

\begin{thm}\label{thm:B}
Let $X$ be a terminal projective variety of dimension $n$ with $K_X$ pseudoeffective and $\chi(X,\OO_X)\neq0$.
\begin{enumerate}
\item[(i)] Assume the existence of good models for klt pairs in dimensions at most $n-1$. If $K_X$ has a singular metric with algebraic singularities and semipositive curvature current, then $\kappa(X,K_X) \geq0$. Moreover, if $K_X$ is hermitian semipositive, then $K_X$ is semiample. 
\item[(ii)] Assume that $K_X$ is nef and $\nu(X,K_X)=1$. Then $\kappa(X,K_X) \geq0$. 
\end{enumerate} 
\end{thm}

We note that if $K_X$ is hermitian semipositive in part (i), and if $\kappa(X,K_X)\geq0$, then $K_X$ is semiample by results of Gongyo and Matsumura \cite{GM14}, hence the main issue to prove here is nonvanishing.

We sketch the proof when $X$ is smooth and $K_X$ is hermitian semipositive. Hence, assume that $\kappa(X,K_X)=-\infty$. Now, Theorem \ref{thm:nonvanishing} implies that
$$h^0\big(X,\Omega^q_X \otimes \OO_X(mK_X)\big)=0$$
for all $q$ and almost all $m>0$, and then Theorem \ref{thm:DPS} yields 
$$h^q\big(X,\Omega^n_X \otimes \OO_X(mK_X)\big)=0$$
for all $q$ and almost all $m>0$, and a fortiori,
$$\chi\big(X,\Omega^n_X \otimes \OO_X(mK_X)\big)=0.$$
In particular, since $\chi\big(X,\Omega^n_X \otimes \OO_X(mK_X)\big)$ is a numerical polynomial in $m$, this last relation holds for every $m$, and in particular for $m={-}1$. This contradicts the assumption $\chi(X,\mathcal O_X) \ne 0$. 

\begin{rem}
When $n=3$, Theorem \ref{thm:B} gives a new proof of (the most difficult part of) nonvanishing in dimension $3$, which avoids Donaldson's theory.
\end{rem}

\subsection{Abundance Conjecture on Ricci flat manifolds}  

One might speculate that given a smooth projective variety $X$ and a nef divisor $L$ on $X$ such that $K_X + L$ is nef, then $K_X + L$ is numerically equivalent to a semiample divisor. This is however not true: indeed, consider $X$ to be $\mathbb P^2$ blown up in $9$ points in general position and set $L = - 2K_X$. Then $K_X+L = -K_X$ will not be semiample. 

However, it is expected that things change when $K_X$ is numerically trivial. The following is the Abundance Conjecture for manifolds with trivial canonical class.

\begin{con} \label{conj:Ab1} 
Let $X$ be a projective manifold with $H^1(X,\mathcal O_X) = 0$  such that $K_X \sim 0$. If $L$ is a nef divisor on $X$, then $L$ is semiample. 
\end{con} 

The reason for assuming $H^1(X, \mathcal O_X) = 0$  is of course to exclude the presence of tori -- otherwise, the conjecture needs to be reformulated to claim that $L$ is numerically equivalent to a divisor $L'$ such that $L'$ is semiample. Note that when $X$ is an abelian variety, the modified conjecture has classically an affirmative answer, see e.g.\ \cite{BL04}. 

\begin{rem} 
Conjecture \ref{conj:Ab1} has been stated by various authors; we refer to \cite{Ver10} for an account. When $X$ is a hyperk\"ahler manifold, Conjecture \ref{conj:Ab1} is  referred to as (a version of) Strominger-Yau-Zaslow (SYZ) conjecture. 
\end{rem} 

A singular version of the conjecture is:

\begin{con} 
Let $X$ be a normal projective klt variety with $H^1(X,\mathcal O_X) = 0$  such that $K_X \sim_\Q 0$. If $L$ is a nef divisor on $X$, then $L$ is semiample. 
\end{con} 

We start discussing Conjecture \ref{conj:Ab1}. When $X$ is a K3 surface, then Conjecture \ref{conj:Ab1} is a simple consequence of Riemann-Roch and the Hodge index theorem. Thus we turn to Calabi-Yau threefolds. 

\subsection{Abundance conjecture: Calabi-Yau threefolds} 

We start with the following observation from \cite{Og93}. 

\begin{pro}\label{pro:semiample}
Let $L$ be a nef divisor on a Calabi-Yau threefold $X$. If $\kappa(X,L) \geq 0$, then $L$ is semiample. 
\end{pro} 

\begin{proof} 
Fix a positive integer $m$ and an effective divisor $D \in |mL|$. Then for a small positive rational number $\varepsilon$ the pair $(X,\varepsilon D)$ is klt, hence $K_X + \varepsilon D$ is semiample by the abundance conjecture for threefolds \cite{KMM94}. Since $K_X$ is trivial, the claim follows.
\end{proof} 

Thus the problem is reduced to showing that $\kappa(X,L) \geq  0$. The Riemann-Roch theorem gives
$$ \chi(X,\mathcal O_X(mL)) = \frac{m}{12} L \cdot c_2(X). $$ 
By a theorem of Miyaoka \cite{Miy87}, $ c_2(X) \in \NEb(X)$, and in particular 
$$L \cdot c_2(X) \geq 0.$$
Moreover $c_2(X) \ne 0$, since otherwise $X$ would be an \'etale quotient of a torus by Yau's theorem, see for instance \cite[IV.4.15]{Kob87}. 

Therefore, if $L \cdot c_2(X) \ne 0$, then $h^0(X, \mathcal O_X(mL))$ or $h^2(X, \mathcal O_X(mL))$ grows at least linearly with $m$. The latter case is ruled out if $\nu(X,L) = 2$ by the Kawamata-Viehweg vanishing \cite[Corollary]{Kaw82}. When $\nu(X,L) = 1$, additional arguments are needed \cite{Og93}; we discuss this also below, using a different approach. Thus, we have:
 
\begin{pro}  \label{prop:c2} 
Let $L$ be a nef divisor on a Calabi-Yau threefold $X$. If $L \cdot c_2(X) \ne 0$, then $L $ is semiample. 
\end{pro} 
 
An important step in our general approach in \cite{LOP16} towards Conjecture \ref{conj:Ab1} is given by the following proposition, which is already stated in \cite{Wi94}.
 
\begin{pro} \label{prop:van} 
Let $L$ be a nef line bundle on a Calabi-Yau threefold $X$. Suppose that $\kappa(X,L) = - \infty$. Then 
$$ H^0\big(X,\Omega^q_X \otimes \sO_X(mL)\big) = 0$$
for all $m \gg 0$ and all $q$. 
\end{pro}
 
We give some ideas of the proof following \cite[Proposition 3.4]{LOP16}. Suppose to the contrary that there exists a number $q$ such that 
$$ H^0\big(X,\Omega^q_X \otimes \sO_X(mL)\big) \neq 0 $$
for infinitely many positive integers $m$. Equivalently, there are infinitely many inclusions $\sO_X(-mL) \to \Omega^q_X$.  Let $\sF \subseteq \Omega^q_X$ be the smallest subsheaf containing the images of all these inclusions, and let $r$ be the rank of $\sF$. Taking determinants and saturation, we obtain a divisor $F$ such that $\sO_X(-F) $ is the saturation of $\det \sF$ in $\bigwedge^r\Omega^q_X$, and such that
$$ H^0\big(X,\mathcal O_X(mL-F)\big) \ne 0$$
for infinitely many $m$. Now consider the induced exact sequence
$$ 0 \to \sO_X(-F) \to \bigwedge^r\Omega^q_X \to \mathcal Q \to 0.$$
Since $\sO_X(-F)$ is saturated,  the sheaf $\mathcal Q$ is torsion free, and hence $c_1(\mathcal Q)$ is pseudoeffective by \cite{CP11,CP15}. As $K_X \sim0$,  we deduce that $F=c_1(\mathcal Q)$, hence the divisor $F$ is pseudoeffective. Thus for all $m$, we find a Weil divisor $N_m\geq0$ such that 
$$ N_m + F \sim mL.$$ 
When $\nu(X,L) = 1$, this can be ruled out using the Hodge index theorem on a hyperplane section and Nakayama's divisorial Zariski decomposition \cite{Nak04}. When $\nu(X,L) = 2 $, we certainly obtain $\rho(X) \geq 3$, which suffices for the most interesting statements below. A more general argument, valid for all values of $\rho(X)$, is given in \cite[Theorem 8.1]{LP16}, see Theorem  \ref{thm:CYn-1} below.

\vskip .2cm  

As a consequence, we have:

\begin{cor} \label{cor:HLF} 
Let $L$ be a nef line bundle on a Calabi-Yau threefold $X$. Suppose that $\kappa(X,L) = - \infty$. Then there exists a number $m_0$ such that for all $m \geq m_0 $ and all $q \geq 0$ we have
$$ H^q(X,mL) = 0.$$
\end{cor} 

Indeed, if $\nu(X,L) = 2$, then as above we already have the vanishing for $q \ne 1$, and the result follows from $\chi(X,mL) = 0$, since $L \cdot c_2(X) = 0$ by Proposition \ref{prop:c2}.

The case $\nu(X,L) = 1$ needs more considerations. Obviously $ H^3(X,\sO_X(mL)) = 0$ for all positive integers $m$ by Serre duality, hence it suffices to show that 
$$ H^2(X,\sO_X(mL)) = 0\quad\text{for large $m$.}$$  
We consider a singular hermitian metric $h$ on the line bundle $\sO_X(mL)$ with semipositive curvature current; such a metric always exist, and we refer to \cite{Dem01} for a detailed discussion. Consider the induced metric $h^m$ on $\sO_X(mL)$, and let $\sI(h^m)$ be the associated multiplier ideal with the corresponding complex subspace $V_m\subseteq X$. It is crucial to observe that $\dim V_m \leq 1$: otherwise $V_m$ would contain a divisor $D$ such that $mL - D$ is pseudoeffective \cite[Lemma 3.7]{LOP16}, which can be ruled out using the Hodge index theorem \cite[Lemma 3.2]{LOP16}. Therefore, we obtain a surjection
$$ H^2\big(X, \sI(h^m) \otimes \sO_X(mL)\big) \to H^2\big(X,\sO_X(mL)\big),$$ 
and it suffices to show that
$$ H^2\big(X, \sI(h^m) \otimes \sO_X(mL)\big) = 0\quad\text{for large $m$.}$$  
Now Proposition \ref{prop:van} gives 
$$H^0(X,\Omega^1_X \otimes \sO_X(mL)) = 0\quad\text{for $m\gg0$}, $$
so a fortiori,
$$H^0\big(X,\Omega^1_X \otimes \sI(h^m) \otimes \sO_X(mL)\big) = 0.$$
By Theorem \ref{thm:DPS}, we obtain a surjective map
$$H^0\big(X,\Omega^1_X \otimes \sI(h^m) \otimes \sO_X(mL)\big) \to  H^2\big(X, \sI(h^m) \otimes \sO_X(mL)\big), $$
which gives the desired vanishing, and proves Corollary \ref{cor:HLF}.

\medskip

The key criterion towards abundance on Calabi-Yau threefolds is the following.

\begin{thm} \label{ample} 
Let $X$ be a Calabi-Yau threefold with $c_3(X) \ne 0$ and let $L$ be a nef divisor on $X$ with $\nu(X,L) =2$. Assume that there is a very ample divisor $H$ and a positive integer $m$  such that for general $D \in | H |$ the following holds:
\begin{enumerate} 
\item the vector bundle $\Omega^1_X(\log D) \otimes \mathcal O_X(mL)$ is nef, and
\item the divisor $L|_D $ is ample. 
\end{enumerate} 
Then $L$ is semiample. 
\end{thm} 

We now discuss the basic ideas of the proof and explain the role of the assumption $c_3(X) \ne 0$. We argue by contradiction and aim to show the following two vanishing statements for $m \gg 0$: 
\begin{equation} \label{van1} 
H^2(X,\Omega^1_X \otimes \sO_X(mL)) = 0  
\end{equation} 
and 
\begin{equation} \label{van2}
H^2(X,\Omega^2_X \otimes \sO_X(mL)) = 0.   
\end{equation} 
These two assertions immediately yield a contradiction. Indeed, \eqref{van1} implies
$$ \chi(X,\Omega^1_X \otimes \sO_X(mL)) = - h^1(X,\Omega^1_X \otimes \sO_X(mL)) \leq 0, $$
and since
$$ \chi(X,\Omega^1_X \otimes \sO_X(mL)) = - \frac{c_3(X)}{2} $$
by Proposition \ref{prop:c2} and by  Riemann-Roch, we obtain $c_3(X) \geq 0$. On the other hand, \eqref{van2}  implies by Serre duality 
$$ H^1(X,\Omega^1_X \otimes \sO_X(-mL)) = 0 $$
for $ m \gg 0$, hence the same argument yields $c_3(X) \leq 0$, hence $c_3(X) = 0.$

The vanishing \eqref{van2} follows easily from a vanishing of Esnault-Viehweg \cite[6.4]{EV92}. As for \eqref{van1}, which is the crucial issue, we use the assumption that the locally free sheaf $\Omega^1_X(\log D) \otimes \mathcal O_X(mL)$ is nef. Since $L|_D$ is ample, the residue sequence
$$ 0 \to \Omega^1_X \to \Omega^1_X(\log D) \to \sO_D \to 0$$
allows to reduce \eqref{van1}  to the vanishing
\begin{equation} \label{van3}
H^2(X,\Omega^1_X(\log D) \otimes \sO_X(mL)) = 0 \quad\text{for }m\gg0.
\end{equation} 
Choosing $m$ sufficiently large,  the locally free sheaf $\Omega^1_X(\log D) \otimes \sO_X(mL) $ is not only nef, but even big. By a standard vanishing theorem for big and nef locally free sheaves, we obtain
$$ H^2(X, \Omega^1_X(\log D) \otimes \sO_X(mL + D)) = 0$$
for $m \gg 0,$ which easily implies \eqref{van3}, and shows Theorem \ref{ample}.

\medskip

Theorem \ref{ample} is used to prove 

\begin{thm} \label{thm:sa} 
Let $L$ be a nef line bundle on a Calabi-Yau threefold $X$ with $c_3(X) \ne 0$. Assume that $\rho(X) = 2$ and $\nu(X,L) = 2$. Then $L$ is semiample. 
\end{thm} 

The condition $\rho(X) = 2$ is used twofold in the verfication of the assumptions of Theorem \ref{ample}. First, we choose an ample divisor $A$ such that $A \cdot c_2(X) $ is minimal under all ample prime divisors on $X$.  Then, using $\rho(X) = 2,$ one shows that given an integral divisor $M$ on $X$ such that $M \sim_\Q aA + bL$ with $a,b\in \Q$ and $a > 0$, then actually $a \geq 1$. This is important for calculations establishing the nefness of $\Omega^1_X(\log D) \otimes \mathcal O_X(mL) $: the proof is rather tricky, and we refer to \cite{LOP16} for details. 

Second, let $H$ be a very ample divisor on $X$ and $D \in | H | $ general. Then $L|_D$ is ample: otherwise, we would obtain a family of curves $(C_t)$ such that $L \cdot C_t = 0$, which would then cover a surface $S\subseteq X$ such that $L^2 \cdot S = 0$. The assumption $\rho(X) = 2$ then implies $S \in | mL |$.

Having verified the assumptions of Theorem \ref{ample}, Theorem \ref{thm:sa} follows. 

\begin{rem} 
Putting things together, let $X$ be a Calabi-Yau threefold with $\rho(X) = 2$ and let $L$ be a nef divisor on $X$ such that $\nu(X,L) = 2$. Suppose that $\kappa (X,L) = - \infty$. Then the following assertions hold:
\begin{enumerate}
\item $ c_3(X)  = 0$,
\item $H^q(X, \Omega^p_X \otimes \sO_X(mL)) = 0$ for $m \gg 0$ and all $p$ and $q.$
\end{enumerate} 
Assertion (2) looks very awkward and one might speculate that (2)  can never happen for a nef line bundle on a say simply connected projective manifold. 
\end{rem} 

\begin{rem} 
If  $\nu(X,L) = 1$, again the critical case is $L \cdot c_2(X) = 0$. In this case, for any irreducible surface $S \subset X$, the restricted  line bundle $L|_S$ is never big. One would expect to find curves $C \subseteq S$ such that $L \cdot C = 0$, however $L|_S$ could be strictly nef; see Subsection \ref{ss:strict} for the definition and discussion of strictly nef line bundles. If there exists a family $(C_t)_{t \in T}$ of curves covering $X$ with $L \cdot C_t$, then $L$ is semiample \cite{LOP16}.  If $\dim T \geq 2$, but the curves cover a surface, then $L$ is semiample, at least when $\rho(X) = 2$.
\end{rem} 

\begin{rem} 
The existence of a semiample non-ample divisor $D \ne 0$ on a Calabi-Yau threefold often has a significant impact on the geometry of $X$. Indeed, let 
$$ \varphi\colon X \to Y$$
be the morphism associated to the linear system $| mD |$ for $m$ sufficiently divisible. Then one of the following cases occurs. 
\begin{enumerate} 
\item $\nu(X,L) = 1$ and $\varphi$ is a K3-fibration or an abelian fibration over $Y \simeq \bP^1$.  Moreover,
$L \cdot c_2(X) = 0$ if and only if $\varphi$ is an abelian fibration.
\item $\nu(X,L) = 2$ and $\varphi$ is an elliptic fibration over a normal projective rational surface $Y$. Moreover, there is a $\Q$-divisor $D\geq0$ such that the pair $(Y,D)$ is klt and $K_Y + D \sim_\Q 0.$ 
\item $\nu(X,L) = 3$ and $Y$ is a normal projective variety with canonical singularities such that $K_Y \sim_\Q 0$. A more detailed structure of $\varphi$ can be given, see \cite{Wi92,Wi93,Wi97}. 
\end{enumerate} 
\end{rem}

\subsection{Abundance conjecture: Calabi-Yau varieties in higher dimensions}

We now discuss nef divisors $L$ on a normal projective klt variety $X$ of any dimension such that $K_X \sim_\Q 0$. In higher dimensions there is a priori a big difference between effectivity of (a multiple of) $L$ and semiampleness of $L$, since abundance in higher dimensions is still wide open. The discussion here follows closely the discussion in Subsection \ref{subsection:4.1}.

Concerning nonvanishing, we start with a generalization of Proposition \ref{prop:van}. 

\begin{thm} \label{thm:CYn-1}
Assume the existence of good models for klt pairs in dimensions at most $n-1$. Let $X$ be a $\mathbb Q$-factorial projective klt variety of dimension $n$ such that $K_X\sim_\mathbb Q0$, and let $L$ be a nef divisor on $X$ such that $\kappa(X,L) = - \infty$. Let $\pi\colon Y\to X$ be a resolution of $X$. Then for every $p\geq1$ we have
$$ H^0(Y,(\Omega^1_Y)^{\otimes p} \otimes \sO_Y(m\pi^*L))=0\quad\text{for all $m\neq0$ sufficiently divisible}.$$
\end{thm} 

The proof follows the argument of the proof of Proposition \ref{prop:van} with the following important new ingredient using the techniques of the Minimal Model Program, see \cite[Theorem 8.2]{LP16}. 

\begin{thm} \label{thm:MMPtwistCY}
Assume the existence of good models for klt pairs in dimensions at most $n-1$. Let $X$ be a $\mathbb Q$-factorial projective klt variety of dimension $n$ such that $K_X\sim_\mathbb Q0$, and let $L$ be a nef divisor on $X$. Assume that there exist a pseudoeffective $\bQ$-divisor $F$ on $X$ and an infinite subset $\mathcal S\subseteq \bN$ such that 
$$N_m+F\sim_\bQ mL$$
for all $m\in\mathcal S$, where $N_m\geq0$ are integral Weil divisors. Then 
$$\kappa(X,L)=\max\{\kappa(X,N_m)\mid m\in\mathcal S\}\geq0.$$
\end{thm} 

Concerning semiampleness, we have the following \cite[Theorem 8.3]{LP16}.

\begin{thm} \label{thm:FormsCY}
Assume the existence of good models for klt pairs in dimensions at most $n-1$. Let $X$ be a $\bQ$-factorial projective klt variety of dimension $n$ such that $K_X\sim_\bQ0$, and let $L$ be a nef divisor on $X$ which is not semiample. Let $\pi\colon Y\to X$ be a resolution of $X$. Then for all $q \geq 1$ and $m \neq 0$ sufficiently divisible we have
$$ h^0\big(Y,(\Omega^1_Y)^{\otimes q} \otimes \sO_Y(m\pi^*L)\big)\leq r,$$
where $r$ is the rank of $(\Omega^1_Y)^{\otimes q}$. In particular, 
$$ h^0\big(Y,\Omega^q_X \otimes \sO_Y(m\pi^*L)\big) \leq \binom{n}{q}.$$
\end{thm} 

In order to exploit Theorem \ref{thm:CYn-1}, using Theorem \ref{thm:DPS} and Riemann-Roch similarly as in the proof of Theorem \ref{thm:B}, one shows:

\begin{cor}\label{cor:nef}
Assume the existence of good models for klt pairs in dimensions at most $n-1$. Let $X$ be a projective klt variety of dimension $n$ such that $K_X\sim_\Q0$, and let $L$ be a nef divisor on $X$. 
\begin{enumerate}
\item[(i)] Assume that $\sO_X(L)$ has a singular hermitian metric with semipositive curvature current and with algebraic singularities. If $\chi(X,\sO_X)\neq0$, then $\kappa(X,L)\geq0$.
\item[(ii)] If $\sO_X(L)$ is hermitian semipositive and if $\chi(X,\sO_X)\neq0$, then $L$ is semiample.
\end{enumerate} 
\end{cor}

The condition $\chi(X,\sO_X) \ne 0$ is necessary, since the conclusion of the corollary is wrong in case of an abelian variety. There might a version of Corollary \ref{cor:nef} without this assumption, stating that $L$ is numerically equivalent to a semiample divisor. However, some of the methods discussed so far fail if $\chi(X,\sO_X) = 0 $, even for a Calabi-Yau $3$-fold. Note that if $\dim X $ is odd and $K_X\sim_\Q 0$, then $\chi(X,\sO_X) = 0$ by \cite[Corollary 6.11]{GKP11}. The case of a hyperk\"ahler manifold will be discussed in the next section, without assuming the MMP in lower dimensions. 

When $\nu(X,L) = 1$, we can say more:

\begin{thm}
Let $X$ be a projective manifold with $K_X \sim_\Q 0$ and let $\mathcal L$ be a nef line bundle on $X$ with $\nu(X,\mathcal L)=1$. Let $\eta\colon \tilde X \to X$ be a finite \'etale cover such that the Beauville-Bogomolov decomposition is of the form
$$\tilde X \simeq T\times \prod X_j,$$ 
where the $X_j$ are even-dimensional Calabi-Yau manifolds or hyperk\"ahler manifolds, and $T$ is an abelian variety. Then there exists a line bundle $\mathcal L'$ numerically equivalent to $\mathcal L$ such that $\kappa (X,\mathcal L') \geq 0$. 
\end{thm}

For the proof we refer to \cite[Theorem 8.9]{LP16}. Observe that here we do not have any inductive assumptions.

\subsection{Abundance conjecture: hyperk\"ahler manifolds}
 
In this section we consider line bundles on compact hyperk\"ahler manifolds, often called irreducible sympectic manifolds. For basics on hyperk\"ahler manifolds we refer to \cite{Hu03}. 

Let $X$ be a compact hyperk\"ahler manifold of dimension $2n$ and let $q$ be the Beauville-Bogomolov-Fujiki form on $X$. Let $L$ be a nef line bundle on $X$. By Fujiki's formula we know that 
$$ c_1(L)^{2n} = \lambda q(L,L),$$
where $\lambda$ is a positive constant. Therefore, $L$ is big if and only if $q(L,L)>0$. If $L$ is big, then $X$ is projective, and an argument similar to that of Proposition \ref{pro:semiample} which invokes the basepoint free theorem instead of the results of \cite{KMM94}, shows that $L$ is semiample. Therefore, when considering abundance on hyperk\"ahler manifolds (projective or not), the only interesting case is when $q(L,L) = 0$. Such line bundles $L$ are often called \emph{parabolic}. 

A form of the Abundance Conjecture \ref{conj:Ab1} for hyperk\"ahler manifolds is now:

\begin{con} \label{con:SYZ} 
Any nef parabolic line bundle on a compact hyperk\"ahler manifold is semiample. 
\end{con} 

The conjecture would imply that a non-trivial nef parabolic line bundle $L$ on a compact hyperk\"ahler manifold $X$ defines a holomorphic surjective map with connected fibers
$$\varphi\colon X \to B$$
to a normal projective variety $B$ such that $0 < \dim B <  \dim X$. Then the results of Matsushita \cite{Mat99,Mat01} show that $\dim B = n$, and in particular, we have $\kappa(X,L)=\nu(X,L) = n$. Moreover, $\varphi$ is a Lagrangian fibration and all smooth fibers are tori. Hence, Conjecture \ref{con:SYZ} allows the following more precise form:

\begin{con} \label{con:para}
Let $L$ be a non-trivial nef parabolic line bundle on a compact hyperk\"ahler manifold $X$ of dimension $2n$. Then $L$ is semiample and $\kappa(X,L) = n$. 
\end{con} 

The first progress towards Conjecture \ref{con:SYZ} is the following important result of Verbitsky \cite{Ver10}.

\begin{thm}
Let $X$ be a projective hyperk\"ahler manifold and let $L$ be a parabolic line bundle on $X$ which is hermitian semipositive. Then $\kappa (X,L) \geq 0$. 
\end{thm} 

Corollary \ref{cor:nef} yields the following in the case of {\it projective} hyperk\"ahler manifolds:

\begin{thm}
Let $X$ be a projective hyperk\"ahler manifold of dimension $2n$, and assume the existence of good minimal models for klt pairs in dimensions at most $2n-1$. Let $L$ be a nef parabolic line bundle on $X$.
 \begin{enumerate}
\item[(i)] If $L$ has a singular hermitian metric with semipositive curvature and with algebraic singularities, then $\kappa(X,L)\geq0$.
\item[(ii)] If $L$ is hermitian semipositive, then $L$ is semiample.
\end{enumerate} 
\end{thm}

When a compact hyperk\"ahler manifold $X$ is not projective, recall that the \emph{algebraic dimension} $a(X)$ is the transcendence degree of the field of meromorphic functions $\C(X)$ over $\C$. In this context, the algebraic dimension of $X$ and $\kappa (X,L)$ are related: 

\begin{con} \label{nonalg} 
Let $X$ be a non-projective compact hyperk\"ahler manifold and let $L$ be a non-trivial nef line bundle on $X$. Then $a(X) = \kappa (X,L)$.
\end{con} 

If $L$ semiample, Conjecture \ref{nonalg}  follows from Matushita's results mentioned above. 

Conjectures  \ref{con:para} and \ref{nonalg} imply 

\begin{con} \label{algdim} 
Let $X$ be a compact hyperk\"ahler manifold of dimension $2n$. Then the algebraic dimension $a(X) $ takes only the values $0$, $n$ and $2n.$ 
\end{con} 

By \cite{COP10}, Conjecture \ref{algdim} holds, provided that any compact K\"ahler manifold $Y$ with $\dim Y \leq 2n - 1$ and $a(Y)= \kappa(Y)=0 $ has a minimal model. In dimension $3$ this assumption is true by \cite{CHP16}, so that Conjecture \ref{algdim} holds in dimension $4$. The paper \cite{COP10} actually proves Conjecture \ref{algdim} in dimension $4$ independently of the K\"ahler MMP in dimension $3$ and also proves that $a(X) \leq n$. Moreover, the following result is shown:

\begin{thm} 
Let $X$ be a non-projective compact hyperk\"ahler manifold of dimension $2n$. Let $L$  be a non-trivial nef line bundle on $X$.
\begin{enumerate} 
\item If $a(X) = n$, then $L$ is semiample.
\item If $L$ is hermitian semipositive, then $a(X) = \kappa(X,L)$.
\end{enumerate} 
\end{thm} 

For further recent results on hyperk\"ahler manifolds, we refer to \cite{AC13,Hw08,GLR13,GLR14}. 

\subsection{Strictly nef line bundles} \label{ss:strict} 

Finally, we quickly address a question which was implicit in several considerations above.

\begin{dfn} 
Let $X$ be a normal projective variety. A divisor $L$ on $X$ is \emph{strictly nef} if $ L \cdot C > 0$ for all irreducible curves $C$ on $X$. 
\end{dfn}

In general, a strictly nef divisor need not be ample: a classical example was found by Mumford, see for instance \cite[p.\ 56]{Ha70}; and if additionally $\kappa(X,L)\geq0$, a counterexample was given by Ramanujam \cite[pp.\ 57-58]{Ha70}. However, it is expected that \emph{adjoint} strictly nef divisors should be ample:

\begin{con} \label{strict} 
Let $X$ be a projective manifold of dimension $n$ and let $L$ be a strictly nef divisor on $X$. Then $K_X + tL$  is ample for $t >  n+1$. 
\end{con} 

For a projective manifold $X$ with $K_X \equiv 0 $, strict nefness and ampleness should therefore coincide. Thus, if $K_X \equiv 0 $ and if $L$ is a nef divisor which is not ample, then Conjecture \ref{strict} predicts that 
$$ L^\perp \cap \partial \NEb(X) $$ 
should contain the class of an irreducible curve. 

Serrano \cite{Se95} established the conjecture for $n = 2$. The paper \cite{CCP08} gives a solution when $\kappa (X) \geq \dim X-2$.  For threefolds, \cite{Se95,CCP08} prove the following:

\begin{thm} 
Let $X$ be a smooth projective threefold and let $L$ be a strictly nef divisor on $X$. Then $K_X + 4L$ is ample, unless possibly when $X$ is a Calabi-Yau threefold 
and $L \cdot c_2(X) = 0$. 
\end{thm} 

Note also that, given $L$ strictly nef, then $K_X + (n+1) L $ is nef by Mori's Cone theorem, and if $K_X + tL$ is not ample for some $t > n+1$, then \cite{Se95} showed that
$$ K_X^j \cdot L^{n-j} = 0 $$
for all $0 \leq j  \leq n$. 

In higher dimensions, not much is known about Conjecture \ref{strict}. Special interesting cases are those when $L = -K_X$ -- then $X$ should be Fano; and when $L = K_X$ -- then $K_X$ should be ample, which is also a consequence of the abundance conjecture. 

\bibliographystyle{amsalpha}

\bibliography{biblio}

\end{document}